\documentclass[11pt]{article}
\usepackage{amssymb}
\usepackage{mathrsfs}
\usepackage{CJK}
\usepackage{bbm}
\usepackage{stmaryrd}
\usepackage{amsmath}
\usepackage{amsfonts}
\usepackage{latexsym,amssymb,amsmath,mathrsfs,amsbsy, amsthm}
\usepackage{amscd}
\usepackage[usenames]{color}
\usepackage{xspace,colortbl}
\usepackage[
pdfstartview=FitH,
CJKbookmarks=true,
 bookmarksnumbered=true,
bookmarksopen=true,
citecolor=blue 
,colorlinks=true
,pdfstartpage=1
]{hyperref} 
\allowdisplaybreaks
\usepackage{pdfsync}

\newtheorem{thm}{Theorem}[section]
\newtheorem{lem}[thm]{Lemma}
\newtheorem{exam}[thm]{Example}
\newtheorem{rem}[thm]{Remark}
\newtheorem{coro}[thm]{Corollary}

\newtheorem{defn}[thm]{Definition}
\newtheorem{prop}[thm]{Proposition}

\newcommand{\mz}{\mathbb{Z}}
\newcommand{\mc}{\mathbb{C}}
\newcommand{\mr}{\mathbb{R}}

\numberwithin{equation}{section} \textheight=8.8in \textwidth=6.28in
\begin{document}

\title{A generalization of Montgomery-Yang correspondence }

\author{\bf Wei WANG \\
}
\date{\small October, 2012}

\maketitle
\begin{abstract}
\small  In this paper, we want to construct a one-to-one correspondence from the set of diffeomorphism classes of spin $d$-twisted homology $\mc P^3$ to the set of isotopy classes of the embedding from $S^3$ to $S^6$, which is a generalization of the Montgomery-Yang correspondence. Furthermore, we will apply this generalized correspondence to prove the existence of free involution on these $d$-twisted homology $\mc P^3$.
\end{abstract}
\footnotetext[1]{\textbf{MSC(2000): 57R55, 57R19}}
\footnotetext[2]{\textbf{Keywords: Montgomery-Yang correspondence, $d$-twisted homology $\mc P^3$, Involution
}}
\section{Introduction}
\subsection{Montgomery-Yang correspondence}
Let $\Pi$ be the set of diffeomorphism classes of homotopy $\mc P^3$ and $\Sigma^{6,3}$ be the set of isotopy classes of the embedding from $S^3$ to $S^6$. By \cite{Haef1} and \cite{Haef}, $\Sigma^{6,3}$ admits an abelian group structure with $\Sigma^{6,3}\cong \mz$.
In their paper \cite{MY}, Montgomery and Yang constructed a one-to-one correspondence from $\Pi$ to $\Sigma^{6,3}$:
\[
\Phi: \Pi \longrightarrow \Sigma^{6,3}.
\]
Furthermore, they constructed an operation $\ast$ on $\Pi$ to make $(\Pi,\ast)$ admit an abelian group structure with unit $\mc P^3$ and proved that this correspondence $\Phi:\Pi \longrightarrow \Sigma^{6,3}\cong \mz$ is also a group isomorphism.

By the classification theorems of closed 1-connected 6-manifold with torsion free integer homology group (cf \cite{Wall}), we may quickly know $\Pi$, as a set, is one-to-one correspondence to $\mz$. But we can not obtain the group structure of $\Pi$ from those classification theorems, which is one of the key points of the paper of Montgomery and Yang \cite{MY}.

\subsection{Main results}
In this paper, we mainly concern with the set $\Pi_{d}^{6}$ of the diffeomorphism classes of spin $d$-twisted homology $\mc P^3$, which is a class of closed 1-connected smooth spin 6-manifolds with cohomology ring $\mz [x,y]/(x^2-dy, y^2),\ degx=2,\ degy=4, \ d\in \mz$.

Follow the idea of $\cite{MY}$, we first construct an operation $\natural$ on $\Pi_{d}^{6}$ to make $(\Pi_{d}^{6},\natural)$ become an abelian group. Second, we
will construct a group isomorphism:
\begin{thm}
$\Phi: \Pi_{d}^{6}\cong \Sigma^{6,3}$.
\end{thm}
As an application, we will prove an existence theorem on $\Pi_{d}^{6}$:
\begin{thm}
For any $M\in \Pi_{d}^{6}$, there exists a free involution on $M$.
\end{thm}
For the involution questions on homotopy projective 3-spaces and $d$-twisted homology $\mc P^3$, we refer to \cite{LL}, \cite{Masu}.

The organization of this paper is as follows. In section 2, we will discuss the basic properties of the element in $\Pi_{6}^{d}$. In section 3, the group structure of $\Pi_{6}^{d}$ will be concerned and Theorem 1.1 will be proved. Finally, we will prove theorem 1.2 in section 4.

\subsection{Notations}
$D^n:=\{x\in \mr^n \ |\ |x|\leqslant 1\}$.\\
$D_{+}^n:=\{(x_1,\cdot\cdot\cdot,x_n)\in D^n \ |\ x_n\geqslant 0\}$.\\
$D_{-}^n:=\{(x_1,\cdot\cdot\cdot,x_n)\in D^n \ |\ x_n\leqslant 0\}$.\\
$E^n:= (D^n)^{\circ}=\{x\in \mr^n \ |\ |x|< 1\}$.\\
$S^{n-1}:=\{x\in \mr^n \ |\ |x|=1\}$.\\

\section{$d$-twisted homology $\mc P^3$}
\subsection{Definition of $d$-twisted homology $\mc P^3$}
\begin{defn} Let $M$ be a closed smooth 1-connected 6-manifold. We call $M$ a $d$-twisted homology $\mc P^3$ if the cohomology ring $H^* (M,\mz)$ of $M$ is isomorphic to:
\[
\mz [x,y]/(x^2-dy, y^2), \ d\in \mz, \ degx=2, \ degy=4.
\]
Since $\mz [x,y]/(x^2-dy, y^2)\cong \mz [x,y]/(x^2+dy, y^2)$, we can always assume $d\geqslant 0.$
\end{defn}

\begin{exam} The first example is 3-dimensional complex projective space $\mc P^3$. It is simply connected and its cohomology ring is isomorphic to $\mz [x]/(x^4),\ degx=2$. By definition, $\mc P^3$ is 1-twisted. Furthermore, any smooth closed manifold $M$ which is homotopic equivalent to $\mc P^3$ is also a 1-twisted homology $\mc P^3$.
\end{exam}

\begin{exam}
$S^2\times S^4$ is a $0$-twisted homology $\mc P^3$.
\end{exam}
\begin{exam} Let $F_d$ be a smooth hypersuface in $\mc P^4$ of degree $d>0$, for example, the Fermat hypersurface: $\{[x_0,x_1,x_2,x_3,x_4]\in \mc P^4| \sum_{i=0}^4 x_i^d =0 \}$. By Lefschetz's hyperplane section theorem (c.f. \cite{MilnorM}) , the pair $(\mc P^4,F_d)$ is 3-connected and $F_d$ admits a connected sum decomposition (cf \cite{Wall} ):
\[
F_d \cong M_d \sharp \frac{b_3(F_d)}{2}S^3 \times S^3.
\]
By calculation, $H^*(M_d,\mz)\cong \mz [x,y]/(x^2-dy, y^2), degx=2, degy=4$. So for any $d>0$, there always exists a $d$-twisted homology $\mc P^3$.
\end{exam}
\begin{rem} In \cite{LW}, Libgober and Wood called a simply connected CW complex $X$ of dimension $6$ a $d$-twisted homology $\mc P^3$ if $H^* (X,\mz)\cong \mz [x,y]/(x^2-dy, y^2), \ d\in \mz, \ degx=2, \ degy=4.$ In our paper, we only concern with closed smooth 6-manifolds, so we make a little change of their definitions. In some references, these manifolds are also called fake projective 3-spaces.
\end{rem}

\subsection{A geometric construction of $d$ twisted homology $\mc P^3$}
Let $j_2:S^3 \hookrightarrow S^2\times S^3$ be the standard embedding. Let $p_1:S^2\times S^3\longrightarrow S^2$ and $p_2:  S^2 \times S^3 \longrightarrow S^3$ be the projection to the first and second factor of $S^2 \times S^3$. For an orientation reversing diffeomorphism $h: S^2 \times S^3 \longrightarrow S^2 \times S^3$ with\\
(1). The map $p_2 \circ h \circ j_2:S^3 \hookrightarrow S^2 \times S^3 \longrightarrow S^2 \times S^3
\longrightarrow S^3$ has degree $-1$.\\
(2). The Hopf invariant of $p_1\circ h \circ j_2:
S^3 \hookrightarrow S^2 \times S^3 \longrightarrow S^2 \times S^3
\longrightarrow S^3$ is $d$.\\
Then we can construct a closed 6-manifold $M(h)$ by gluing two copies of $S^2\times D^4$ along their boundaries through this orientation reversing diffeomorphism $h$:
\[
M(h):= (S^2 \times D^4)\cup_h (S^2\times D^4).
\]

\begin{lem}
$M(h)$ is a $d$-twisted homology $\mc P^3$.
\end{lem}
\begin{proof}
By Van-Kampen's theorem, $M(h)$ is simply connected. The Meyer-Vietoris exact sequence of $M(h)$ shows that the group structure of $H^*(M(h),\mz)=\oplus_{i=0}^{6}H^i (M(h),\mz)$ is isomorphic to $\mz\oplus\mz x\oplus \mz y \oplus \mz z$, with $degx=2, degy=4$, and $degz=6$. By Poincar\'{e} duality, we can choose $x\cup y$ as a generator of $\mz z$ and $H^*(M(h),\mz)\cong \mz\oplus\mz x\oplus \mz y \oplus \mz (x\cup y)$.

Next we need to find the relation between $x\cup x$ and $y$, which also determines the ring structure of $H^*(M(h),\mz)$. Since $S^2= E^2\cup pt$,
\[
M(h)=(S^2 \times D^4)\cup_h (S^2 \times D^4)= (E^2 \times E^4)\cup ((pt\times D^4)\cup_{hj_2}(S^2\times D^4 )).
\]
Therefore,
$$M(h)-E^6 \cong M(h)-E^2\times E^4\cong (pt\times D^4)\cup_{hj_2}(S^2\times D^4 )\simeq D^4\cup_{p_1 h j_2} S^2$$
and we get the ring isomorphism $H^*(M(h)-E^6,\mz)\cong H^*(D^4\cup_{p_1 h j_2} S^2,\mz)$. By the definition of Hopf invariant, the cohomology ring $H^*(D^4\cup_{p_1 h j_2} S^2,\mz)\cong \mz[x,y]/(x^2-dy,xy,y^2)$ with $degx=2, degy=4$, where $d$ is the Hopf invariant of the map $p_1 h j_2$.

On the other hand, the restriction map $H^m(M(h),\mz)\longrightarrow H^m (M(h)-E^6,\mz)$ is isomorphic when $m\neq 6$ and we get the relation $x\cup x= dy$ in $H^*(M(h),\mz)$. Finally, combine the group structure of $H^*(M(h),\mz)$ and the relation $x\cup x =d y$, we see the cohomology ring $H^*(M(h),\mz)\cong \mz[x,y]/(x^2-dy,y^2), degx=2, degy=4$.
\end{proof}

What about the converse? First, following \cite{MY}, let $M$ be a $d$-twisted homology $\mc P^3$ and $i:S^2\hookrightarrow M$ be an embedding, we call $i:S^2\hookrightarrow M$ a \textbf{primary embedding} if $i_*:H_2(S^2,\mz)\longrightarrow H_2(M,\mz)$ maps the generator $[S^2]\in H_2(S^2,\mz)$ to the generator of $H_2(M,\mz)$ which represents the orientation of $M$. 

\begin{exam}
For $M(h)=(S^2 \times D^4)\cup_h (S^2\times D^4)$ and each copy $S^2\times D^4$, the inclusion $S^2\hookrightarrow S^2\times D^4\subset M(h)$ induces isomorphism $H_2(S^2,\mz)\cong H_2(S^2\times D^4,\mz)\cong H_2(M(h),\mz)$. Since the attaching map $h$ is orientation-reversing, we can make these two copies $S^2\times D^4\subset M(h)$ preserve the orientation. Then these two inclusions $S^2\hookrightarrow S^2\times D^4\subset M(h)$ are primary embeddings.
\end{exam}
\begin{lem}
Let $M$ be a $d$-twisted homology $\mc P^3$ and $i:S^2 \longrightarrow M$ be a primary embedding. The normal bundle $\eta_i$ of $i$ is trivial if and only if $M$ is spin, i.e. $w_2 (M)=0$.
\end{lem}
\begin{proof}
The normal bundle bundle $\eta_i$ is trivial if and only if the second obstruction $\sigma_2 \in H^2(S^2, \pi_1 (SO(4)/SO(1)))=H^2(S^2,\pi_1(SO(4)))=H^2(S^2,\mz/2\mz)$ is zero. From \cite{Milnor}, this obstruction $\sigma_2$ is equal to the second Stiefel-Whitney class $w_2(\eta_i)$, which is also equal to $i^* w_2 (M)$.

Since $i$ is a primary embedding and $\pi_1(M)=1$, the restriction map $i^*: H^2(M.\mz/2\mz)\longrightarrow H^2(S^2,\mz/2\mz)$ is an isomorphism. So $\eta_i$ is trivial if and only if $w_2(M)=0$. 
\end{proof}
\begin{coro}
$M(h)$ is spin.
\end{coro}
\begin{proof}
$M(h)=(S^2 \times D^4)\cup_h (S^2\times D^4)$, choose one primary embedding $S^2\hookrightarrow S^2\times D^4\subset M(h)$. The normal bundle of $S^2$ is trivial.
\end{proof}

\begin{prop}
Let $M$ be a $d$-twisted homology $\mc P^3$, then $M$ is spin if and only if
\[
M\cong (S^2 \times D^4)\cup_h (S^2\times D^4),
\]
where $h: S^2 \times S^3 \longrightarrow S^2 \times S^3$ is an orientation reversing diffeomorphism with
(1). $p_2 \circ h \circ j_2$ has degree $-1$, (2). The Hopf invariant of $p_1\circ h \circ j_2$ is $d$.
\end{prop}
\begin{proof}
Only one direction needs to be proved. Assume $M$ is spin. According to Whitney's embedding theorem, we can represent the generator of $H_2(M,\mz)$ which represents the orientation class by a primary embedding $i:S^2 \hookrightarrow M$ and any two such primary embeddings are isotopic. From Lemma 2.8, the normal bundle of $i:S^2\hookrightarrow M$ is trivial and we have an embedding $f:S^2\times D^4\longrightarrow M$ with $f(x,0)=i$.

Consider the complement $M-f(S^2\times E^4)\simeq M-i(S^2)$, $\pi_1(M-S^2)=1$ and $H_*(M-i(S^2),\mz)\cong H_*(S^2,\mz)$, which implies $M-i(S^2)\simeq S^2$. Since $\pi_2(M-i(S^2))\cong \pi_2(M)$, we can repersent the generator of $\pi_2(M-i(S^2))\cong \pi_2(M)$ by a primary embedding $j:S^2\hookrightarrow M-f(S^2\times E^4)\subset M$ and $j:S^2\hookrightarrow M-f(S^2\times E^4)$ is also a homotopy equivalence by Whitehead theorem. By the standard technique of $h$-cobordism theorem, we see $M-f(S^2\times E^4)$ is diffeomorphic to the normal bundle of $j:S^2 \hookrightarrow M$, which is a trivial bundle, and we obtain a diffeomorphsim $g: S^2\times D^4\cong M-f(S^2\times E^4)$ with $g(x,0)=j$.

Finally, we get $M\cong f(S^2\times D^4)\cup_{f(S^2\times S^3)} g(S^2\times D^4)\cong (S^2\times D^4)\cup_h (S^2\times D^4)$, where $h=g^{-1}f:S^2 \times S^3 \longrightarrow S^2\times S^3$ is the attaching diffeomorphsim. Furthermore, the embeddings $f$ and $g$ are orientation preserving and the two copies of $S^2\times D^4$ of $(S^2\times D^4)\cup_h (S^2\times D^4)$ are also orientation preserving, the attaching map $h$ has to be orientation reversing.

For the degree of $p_2 h j_2$, we see the primary embeddings $i=f(x,0)$ and $j=g(x,0)$ are isotopic to each other, which implies $h_*:H_2(S^2\times S^3,\mz)\longrightarrow H_2(S^2\times S^3,\mz)$ is equal to identity. Then on $H_3(S^2\times S^3,\mz)$, $h=-id$ because $h$ is orientation reversing and we get the degree of $p_2hj_2$ is $-1$. By lemma 2.6, the Hopf invariant of $p_1hj_2$ is $d$.
\end{proof}

\begin{prop}
If $d$ is odd, any $d$ twisted homology $\mc P^3$ is spin.
\end{prop}
\begin{proof}
When $d$ is odd, $H^*(M,\mz/2\mz)\cong \mz/2\mz [x]/(x^4)$ with $degx=2$ and the homomorphism $Sq^i:H^{6-i}(M,\mz/2\mz)\longrightarrow H^6(M,\mz/2\mz)$ is zero. Indeed, $Sq^1,\ Sq^3=0$ because $H^{odd}(M,\mz/2\mz)=0$ and for the generator $x^2\in H^4(M,\mz/2\mz)$, $Sq^2(x^2)=2x^3=0$. $Sq^*=0$ implies the total Wu class $v(M)$ of $M$ is equal to 1 and the total Stiefel-Whitney class $w(M)=Sq(v(M))=1$.
\end{proof}

\section{Group structure of $\Pi_{d}^{6}$}
\subsection{The clutching diffeomorphism $f_d$}
Following \cite{MY}, let $\mr^4$ be the quaternion field and there exists a differential $S^1$ action on $\mr^4$ by left multiplication. Under this action, there is a free differential $S^1$ action on $S^3\subset \mr^4$ whose quotient space $S^3/S^1$ is just $S^2$. For convenience, we identify $S^2$ by the quotient space $S^3/S^1:=\{S^1x| x\in S^3\subset \mr^4\}$.

We define an orientation-reversing diffeomorphism of $S^2\times S^3$ by:
\[
f:S^2\times S^3\longrightarrow S^2\times S^3
\]
\[
(S^1x,y)\mapsto (S^1xy, y^{-1}),
\]
here $y^{-1}$ is the inverse of $y$ in the quaternion field $\mr^4$. From \cite{MY}, we know $f\circ f=id$; the degree of the composition of $p_2\circ f\circ j_2$ is $-1$; and the Hopf invariant of $p_1\circ f\circ j_2$ is 1.

Let $L:\mr^4\longrightarrow \mr^4$ be a reflection along the hyperplane $\{(x_1,x_2,x_3,x_4)\in \mr^4\ |\ x_4=0\}$ and $L|_{S^3}$ induces an involution of $S^3$ with degree $-1$, i.e. $L|_{S^3}\circ L|_{S^3}=id$ and $deg L|_{S^3}=-1$. Define an involution $l:S^2\times S^3\longrightarrow S^2\times S^3$ by $l(x,y)=(x,Ly)$.

Then we can define our clutching maps $f_d,\ d\geqslant 0$. Let $f_0:=l$, and for $d>0,$ define:
\[
f_d:= f\circ l \circ f_{d-1}.
\]
Since $l\circ l= id$ and $f\circ f=id$, we see $f_1=f$ and $f_d\circ f_d=id$.
Furthermore, by the properties of $f$ and the construction of $f_d$ , these $f_d$ admit special properties:
\begin{lem} $f_d$ is an orientation-reversing diffeomorphism with:
(1). $deg(p_2\circ f_d\circ j_2)=-1$.\\
(2). The Hopf invariant of $p_1 \circ f_d \circ j_2$ is equal to $d$.
\end{lem}
\begin{proof}
First, $l$ and $f$ are orientation-reversing and $f_d=f\circ l\circ f_{d-1}$, $f_d$ is also orientation reversing.

Second, $H_3(S^2\times S^3,\mz)=\mz$ and we get $deg(p_2\circ f_d\circ j_2)=deg(p_2\circ f\circ j_2)\cdot deg(p_2\circ l\circ j_2)\cdot deg(p_2\circ f_{d-1}\circ j_2)=deg(p_2\circ f_{d-1}\circ j_2)=deg(p_2\circ f\circ j_2)=-1$.

Finally, $\pi_3(S^2\times S^3)=\mz a \times \mz b$, where $a$ and $b$ are the generators of $\pi_3(S^2)$ and $\pi_3(S^3)$. For $f:S^2\times S^3\rightarrow S^2\times S^3,$ we know $\pi_3(f)(a)=\pi_3(l)(a)=a$, $\pi_3(l)(b)=-b$ and $\pi_3(f)(b)=a-b$. So we obtain $\pi_3 (f_d)(b)=da-b$ and the Hopf invariant of $p_1 \circ f_d \circ j_2$ is just equal to $d$, which is the value of $\pi_3(p_2 \circ f_d)(b)=da$.
\end{proof}
\subsection{The group structure of $\Pi_{d}^{6}$}
Denote $\Pi_{d}^{6}$ by the set of all diffeomorphism classes of spin $d$-twisted homology $\mc P^3$. We assume $d\geqslant 0$ because $\Pi_{d}^{6}=\Pi_{-d}^{6}$.

Let $M_1$ and $M_2$ be two elements in $\Pi_{d}^{6}$. Since $M_1$ and $M_2$ are spin, we can choose two embeddings $h_1:S^2 \times D^4 \subset M_1$ and $h_2:S^2 \times D^4 \subset M_2$ such that $h_1(x,0):S^2\hookrightarrow M_1$ and $h_2(x,0):S^2\hookrightarrow M_2$ are primary embedings of $M_1$ and $M_2$. We define the operation $\natural$ by:
\[
M\natural N:= (M_1-h_1(S^2\times E^4))\cup_{h_2 f_dh_1^{-1}}(M_2- h_2 (S^2\times E^4) ).
\]
Since any primary embeddings of $M_i$ are isotopic as well as their normal bundles, the diffeomorphism class of $M_1 \natural M_2$ is invariant under the variant choices of primary embeddings and their normal bundles. So the operation $\natural$ is well-defined. Furthermore, we have:
\begin{prop}
$M_1 \natural M_2\in \Pi_{d}^{6}.$
\end{prop}
\begin{proof}
By proposition 2.10, there exist $k_i:S^2\times D^4\subset M_i,\ i=1,2$ such that $M_i - h_i(S^2\times E^4)=k_i(S^2 \times D^4)$. The diffeomorphism $h_i^{-1}k_i:S^2\times S^3 \longrightarrow S^2\times S^3$ satisfy: (1). the degree of $p_2 (h_i^{-1}k_i)j_2$ is $-1$, (2) the Hopf invariant of $p_1(h_i^{-1}k_i)j_2$ is equal to $d$, $i=1,\ 2$.

Then $M_1\natural M_2$ is diffeomorphic to $(S^2\times D^4)\cup_{\lambda}(S^2\times D^4)$, where $\lambda= (k_2^{-1}h_2)f_d(h_1^{-1}k_1)$. We see the diffeomorphism $\lambda$ satisfies: (1) the degree of $p_2 \lambda j_2$ is $-1$, (2). the Hopf invariant of $p_1\lambda j_2$ is equal to $d$. By lemma 2.6 and 2.9, $M_1\natural M_2 $ is a spin $d$-twisted homology $\mc P^3$.
\end{proof}
\begin{thm}
$(\Pi_{d}^{6},\natural)$ is an abelian group.
\end{thm}
\begin{proof}
For any $M\in \Pi_{d}^{6}$, by lemma 2.8, $M\cong M(h)= (S^2\times D^4)\cup_h (S^2\times D^4)$.
We have $M(f_d)\natural M\cong (S^2\times D^4)\cup_{hf_df_d}(S^2\times D^4)=M(h)\cong M$, since $f_df_d=id$. And we have $M\natural M(f_d)\cong M$ for the same reason. For any $M\cong M(h)$, $M(f_d h^{-1} f_d)$ is also a $d$-twisted homology $\mc P^3$ and $M(h)\natural M(f_d h^{-1} f_d)=M(f_d)$. It is not difficulty to see $M_1\natural M_2 \cong M_2\natural M_1$ and $\natural$ is associative. Therefore, $(\Pi_{d}^{6},\natural )$ is an abelian group with unit $M(f_d)$.
\end{proof}

\begin{rem}
When $d=1$, $\Pi_1^{6}$ is just the set of diffeomorphism classes of homotopy $\mc P^3$ and this group $(\Pi_{1}^{6},\natural)$ is just the group defined by Montgomery and Yang in \cite{MY} with unit $M(f_1)\cong \mc P^3$.
\end{rem}

\section{The correspondence $\Phi: \Pi_{d}^{6} \longrightarrow \Sigma^{6,3}$}
\subsection{$\Sigma^{6,3}$}
Let $\Sigma^{6,3}$ be the set of isotopy classes of embedding of $S^3$ in $S^6$. $\Sigma^{6,3}$ is an abelian group whose operation is defined as follows (cf \cite{Haef}): let $f_1,f_2$ be two embedding classes of $S^3$ in $S^6$. Then $f_1$ and $f_2$ is isotopic to $g_1$ and $g_2$ such that:\\
(1). $g_1 |D_{-}^3$ is the standard embedding of $D_{-}^3\hookrightarrow D_{-}^6$ and $g_2 |D_{+}^3$ is the standard embedding of $D_{+}^3\hookrightarrow D_{+}^6$. (2). $g_1(D_{+}^3)\subset D_{+}^6$, $g_2 (D_{-}^3)\subset D_{-}^6$.
Define $f_1 + f_2$ by:
\[
f_1 + f_2 (x)=\left\{
\begin{array}{ll}
g_1(x), & x\in D_{+}^3\\
g_2 (x), & x\in D_{-}^3
\end{array}
\right..
\]
Under this operation, Haefliger \cite{Haef} proved that:

\begin{thm}[Haefliger]
$(\Sigma^{6,3},+)\cong \mz$.
\end{thm}
\subsection{$\Phi: \Pi_{d}^{6}\longrightarrow \Sigma^{6,3}$}
Now we construct a map from $\Pi_{d}^6$ to $\Sigma^{6,3}$:

First, for any $M\in \Pi_{d}^{6}$, choose an arbitrary embedding $h:S^2\times D^4\subset M$ such that $h(x,0)$ is a primary embedding and we also get an embedding $k:S^2\times D^4\cong M-h(S^2\times D^4)\subset M$ with primary embedding $k(x,0)$. Let $\alpha:=f_d h^{-1}k:S^2\times S^3\longrightarrow S^2\times S^3$ be the self diffeomorphism of $S^2\times S^3$, where $f_d$ was defined in subsection 3.1. We see this clutching map satisfies $deg(p_2\alpha j_2)=1$ and $p_1\alpha j_2\simeq pt$. Then one can construct a simply connected closed manifold:
\[
P:=(M-h(S^2\times E^4))\cup_{f_d h^{-1}}(D^3\times S^3)\cong (S^2\times D^4)\cup{\alpha}(D^3\times S^3)
\]
with an embedding $i:S^3\hookrightarrow D^3\times S^3\subset P$. We see $P$ is simply-connected and $H^*(P,\mz)\cong H^*(S^6,\mz)$ by Meyer-Vietoris exact sequence. $P$ is diffeomorphic to $S^6$ because $S^6$ is the only homotopy sphere in dimension 6. Thus we define:
\[
\Phi(M):=[P,i]
\]
Since any primary embeddings and their normal bundles are isotopic, $\Phi$ is well-defined. On the other hand, $\Pi_{d}^{6}$ and $\Sigma^{6,3}$ both admit group structures. We can prove:
\begin{prop}
$\Phi$ is a group homomorphism.
\end{prop}
\begin{proof}
Follow the method of \cite{MY}, let $M,\ N \in \Pi_{d}^{6}$ and $h:S^2\times D^4\subset M$ and $k:S^2\times D^4\subset N$ be two embeddings such that $h(x,0)$ and $k(x,0)$ are primary embeddings. We know there exist two embeddings $h':S^2\times D^4\cong M-h(S^2\times E^4)$ and $k':S^2\times D^4 \cong N-k(S^2\times E^4)$ with primary embeddings $h'(x,0)$ and $k'(x,0)$. By the definition of $\Phi$, $\Phi(M)$ is the embedding:
\[
j_1: S^3\hookrightarrow D^3\times S^3\subset (M-h(S^2\times E^4))\cup_{f_d h^{-1}}(D^3\times S^3)\cong (S^2\times D^4)\cup_{\alpha}(D^3\times S^3),
\]
and $\Phi(N)$ is the embedding:
\[
j_2: S^3\hookrightarrow D^3\times S^3\subset (N-k(S^2\times E^4))\cup_{f_d h^{-1}}(D^3\times S^3)\cong (S^2\times D^4)\cup_{\beta}(D^3\times S^3),
\]
where $\alpha=f_d h^{-1}h'$ and $\beta=f_d k^{-1}k'$ are the self diffeomorphism of $S^2\times S^3$.
We continue to make $\alpha$ and $\beta$ diffeotopic to the diffeomorphism such that $\alpha$ is identity on $S^2\times D_{-}^{3}$ and $\beta$ is identity on $S^2\times D_{+}^{3}$. Certainly, $S^2\times D_{+}^{4}\cup D^3\times D_{+}^{3}\cong D_{+}^{6}$ and $S^2\times D_{-}^{4}\cup D^3\times D_{-}^{3}\cong D_{-}^{6}$.

By the definition of the operation $\natural$: $M\natural N=M-h(S^2\times E^4)\cup_{k'f_dh^{-1}}N-k'(S^2\times E^4)\cong h'(S^2\times D^4)\cup_{k'f_dh^{-1}} k(S^2\times D^4)$. $\Phi(M\natural N)$ is the embedding:
\[
j:S^3\hookrightarrow D^3\times S^3\subset (M\natural N -k(S^2\times E^4))\cup_{f_d k^{-1}} D^3\times S^3
\]
\[
\cong (S^2\times D^4)\cup_{f_d k^{-1}k'f_dh^{-1}h'}(D^3\times S^3 )\cong (S^2\times D^4)\cup_{\beta \alpha}(D^3\times S^3 ).
\]
We see $j$ is just the sum of $j_2$ and $j_1$ under $\Sigma^{6,3}$ and we obtain $\Phi(M\natural N)=\Phi(M)+\Phi(N)$.
\end{proof}

\begin{thm}
$\Phi:\Pi_{d}^{d}\longrightarrow \Sigma^{6,3}$ is isomorphic.
\end{thm}
\begin{proof}
For any embedding $j:S^3\hookrightarrow S^6$, the normal bundle of $j$ is trivial and there exists an embedding $k:S^3\times D^3\subset S^6$ with $k(x,0)=j$. The complement $S^6 -k(S^3\times E^3)$ is homotopy equivalent to $S^2$ and by $h$-cobordism theorem, we have a diffeomorphism $h:S^2\times D^4\cong S^6 -k(S^3\times E^3)$. So we have a decomposition $S^6=h(S^2\times D^4)\cup k(D^3\times S^3)\cong S^2\times D^4\cup_{kh^{-1}}D^3\times S^3$, where $kh^{-1}$ is the self diffeomorphism of $S^2\times S^3$.
By the technique of \cite{MY} (cf p494), one can modify $k$ to make the attaching map $kh^{-1}$ admit $degp_2(kh^{-1})j_2=1$ and $p_1(kh^{-1})j_2\simeq pt$.

Define $M(\psi):=(S^2\times D^4)\cup_{\psi}(S^2\times D^4)$, where $\psi=f_d (kh^{-1})$. Then we see $M(\psi)\in \Pi_{d}^{6}$ and $\Phi(M)=[S^6,j]\in \Sigma^{6,3}$. So $\Phi$ is surjective. Next, we want to prove $\Phi$ is injective.

For $M_1=(S^2\times D^4)\cup_{\lambda_1}(S^2\times D^4)$ and $M_2=(S^2\times D^4)\cup_{\lambda_2}(S^2\times D^4)$, if the embeddings $\Phi(M_1)=j_1:S^3\hookrightarrow D^3\times S^3\subset (S^2\times D^4)\cup_{f_d\lambda_1}(D^3\times S^3)$ and $\Phi(M_2)=j_2:S^3\hookrightarrow D^3\times S^3\subset (S^2\times D^4)\cup_{f_d\lambda_2}(D^3\times S^3)$ are isotopic. Then by the tubular neighborhood theorem of $j_1,j_2$, after an isotopy which does not affect the diffeomorphism type, there exists a diffeomorphism $G:S^2\times S^3\longrightarrow S^2\times S^3, \ G(x,y)=(g(y)x,y)$ such that $\lambda_1=\lambda_2 G$, where $g:S^3\longrightarrow SO(3)$. Since the Hopf invariant of $p_1\lambda j_2$ and $p_1\lambda_2 j_2$ are the same, which implies $p_2Gj_2=p_2 g=0\in \pi_3 (S^2)$. Since $\pi_3(SO(3))\cong \pi_3(SO(3)/SO(2))=\pi_3(S^2)$, we see $g\simeq pt$ and $\lambda_1$ is isotopic to $\lambda_2$ and $M_1\cong M_2$.
\end{proof}

\begin{rem}
From Wall's paper \cite{Wall}, we see for $M\in \Pi_{d}^{6}$ with $H^*(M,\mz)=\mz[x,y]/(x^2-dy, y^2), degx=2,degy=4$, the first Pontrjagin class $p_1(M)$ of $M$ is equal to $(4d-24\Phi(M))y$. If $\Phi(M_1)=\Phi(M_2)$, by the classification theorem of simply connected closed spin 6-manifolds with torsion free integral homology (cf \cite{Wall}), $M_1\cong M_2$.
\end{rem}
\subsection{Homotopy type of $\Pi_{d}^{6}$}
When $d=1$, $\Pi_{1}^{6}$ is the set of diffeomorphism classes of homotopy $\mc P^3$, what about the set $\Pi_{d}^{6}$, $d>1$ ?  In their paper \cite{LW}, Libgober and Wood proved:
\begin{thm}[Libgober, Wood]
If $n=2m+1$ and if $d$ has not divisors less than $m+2$, then any two $d$-twisted homology $\mc P^3$ are homotopy equivalent.
\end{thm}
When $n=3$ and $d$ is odd, we get:
\begin{coro}
If $d$ is odd, for any $M_1$, $M_2\in \Pi_{d}^{6}$, $M_1\simeq M_2$.
\end{coro}

When $d$ is even and for $M_1$, $M_2\in \Pi_{d}^{6}$, from \cite{LW} section 9 and \cite{Wall} section 7 we know $M_1\simeq M_2$ if and only if their first Pontrjagin classes satisfy $p_1(M_1)\equiv p_1 (M_2) \ mod(48)$, where $p_1(M_i)\in H^4(M_i,\mz)=\mz,\ i=1,2.$ Furthermore, $p_1(M_i)=4d-24\Phi(M_i)$, where $\Phi:\Pi_{d}^{6}\longrightarrow \Sigma^{3,3}\cong \mz$ is the isomorphism constructed above (cf \cite{Wall} section 6). We obtain:
\begin{prop}
If $d$ is even, for any $M_1,\ M_2\in \Phi^{-1}(2\mz)$ or $\Phi^{-1}(2\mz +1)$, $M_1\simeq M_2$.
\end{prop}

\section{Application to the free involution of $d$-twisted homology $\mc P^3$}
In \cite{LL}, Bang-he Li and Zhi L\"{u} proved
\begin{thm}[Li-L\"{u}]
For any embedding $i:S^3\hookrightarrow S^6$, there exists an involution $\sigma$ on $S^6$ such that the fixed point is $i(S^3)$.
\end{thm}
In this section, we will use their result and the correspondence $\Phi$ constructed above to prove:
\begin{thm}
For any $M\in \Pi_{d}^{6}$, there exists a free involution on $M$.
\end{thm}
\begin{proof}
For any embedding $i:S^3\hookrightarrow S^6$, choose such an involution $\sigma:S^6\longrightarrow S^6$ with fixed point $i(S^3)$. By the equivariant tubular neighborhood theorem, there exists a $\mz/2\mz$ equivariant embedding $h: S^3\times D^3\longrightarrow S^6$ with $\sigma h(x,y)=h(x,-y)$, where the action of $\mz/2\mz\curvearrowright S^3\times D^3$ is $(x,y)\mapsto (x,-y)$. On the other hand, $\sigma$ is free on $S^6-h(S^3\times E^3)\cong S^2\times D^4$. So on $S^6=(S^2\times D^4)\cup_{\lambda}(S^3\times D^3)$, the involution restricted on the part of $S^3\times D^3$ can be $\sigma|S^3\times D^3:(x,y)\mapsto (x,-y)$. By the technique of \cite{MY}, we can make the diffeomorphism $\lambda$ satisfy $deg(p_2\lambda j_2)=1$ and $p_1\lambda j_2\simeq pt$.

Define a free involution $\iota$ on $S^2\times D^4$ by $\iota (x,y)=(-x,y)$. We see $\iota\lambda=\lambda \sigma $ and $\iota f_d=f_d\iota$ on $S^2\times S^3$. For the manifold, $(S_a^2\times D^4)\cup_{f_d \lambda}(S_b^2\times D^4)$, where $S_a^2, S_b^2$ are 2-spheres with label $a$ and $b$, we see $\Phi((S_a^2\times D^4)\cup_{f_d \lambda}(S_b^2\times D^4))=i:S^3\hookrightarrow S^6$.  We define a free involution $\delta$ on $(S_a^2\times D^4)\cup_{f_d \lambda}(S_b^2\times D^4)$ by $\delta|S_a^2\times D^4=\sigma $ and $\delta|S_b^2 \times D^4=\iota$. On the boundary $S^2\times S^3$ of these two copies, we see $\delta f_d \lambda=\iota f_d\lambda=f_d\iota\lambda=f_d \lambda\sigma=f_d\lambda\delta$. So $\delta$ is well-defined and is also a free involution.

In section 4, we've proved that $\Phi$ is an isomorphism. Thus, we conclude that for any $M\in \Pi_{d}^{6}$, there always exists a free involution.
\end{proof}

{\bf Acknowlegement} I would like to thank Professor Zhi L\"{u} for introducing this question to me and for many helpful comments and suggestions.

\noindent{\small  College of Information Technology,
Shanghai Ocean University,
999 Hucheng Huan Road, 201306, Shanghai, China.}\\
\noindent{  \emph{Email address:} \verb"weiwang@amss.ac.cn"}


\end{document}